\documentclass{article}
\usepackage[paperwidth=7in, paperheight=10in, margin=.875in]{geometry}
\usepackage[backref,colorlinks,linkcolor=red,anchorcolor=green,citecolor=blue]{hyperref}
\usepackage{amsfonts,amssymb}
\usepackage{amsmath,amsrefs,amsthm}
\usepackage{graphicx}
\usepackage{cite}
\usepackage{enumerate}
\newtheorem{theorem}{Theorem}[section]

\newtheorem{remark}{Remark}[section]

\newtheorem{definition}{Definition}[section]

\usepackage{pgf,tikz,pgfplots,mathrsfs}
\usetikzlibrary{arrows}
\usepackage{hyperref}

\DeclareMathOperator{\tv}{Tot.Var.}
\DeclareMathOperator{\sign}{sign}

\allowdisplaybreaks

\begin{document}

\title{The Cauchy Problem for a non Strictly Hyperbolic $3\times3$
  System 
  of Conservation
  Laws Arising in Polymer Flooding}

\author{Graziano Guerra\thanks{Department of Mathematics and its Applications,
  University of Milano - Bicocca, Italy, (graziano.guerra@unimib.it).}
          \and Wen Shen\thanks{Mathematics Department, Pennsylvania State University, USA, (wxs27@psu.edu).}}

        \pagestyle{myheadings}
        \markboth{Polymer Flooding}{G.~Guerra and W.~Shen}
        \maketitle

\begin{abstract}
  We study the Cauchy problem of a $3\times 3$ system of conservation laws 
  modeling two--phase flow of polymer flooding in rough porous media  
  with possibly discontinuous permeability function. 
  The system loses strict hyperbolicity in some regions of the domain where the eigenvalues
  of different families coincide, and BV estimates are not available in general. 
  For a suitable $2\times 2$ system,
  a singular change of variable introduced by Temple~\cite{T82,IT86}
  could be effective to control the total variation~\cite{WSCauchy}.
  An extension  of this technique can be applied to a $3\times 3$ system only under 
  strict hypotheses on the flux functions~\cite{CocliteRisebro}. 
  In this paper, through an adapted front tracking algorithm we prove the existence
  of solutions for the Cauchy problem under  mild assumptions on the
  flux function,  using a compensated compactness argument. 


\end{abstract}

\section{Introduction}

We consider a simple model for polymer flooding in two phase flow 
in rough media
\begin{equation}\label{eq1}
\begin{cases}
 \partial_{t} s + \partial_{x}f(s,c,k) &=~0 , \\
\partial_{t} [cs] + \partial_{x}[cf(s,c,k)] &=~0,\\
 \partial_{t}k &=~0,
\end{cases}
\end{equation}
associated with the initial data
\begin{equation}
  \label{eq:init_data}
  \left(s,c,k\right)(0,x)=\left(\bar s,\bar c,\bar k\right)(x),\quad x\in\mathbb{R}.
\end{equation}
Here, the unknown vector is $(s,c,k)$, where 
$s$ is the saturation of water phase, $c$ is the fraction of polymer dissolved in water,
and $k$ denotes the permeability of the porous media.
We see that $k$ does not change in time, $k(t,x) =\bar k(x)$, 
and the initial data $\bar k(\cdot)$ might be discontinuous.

We neglect both the adsorption of polymers in the porous media and
the gravitational force, where the solution to the Riemann problem becomes more complex. 
For such Riemann solvers, see~\cite{JohWin,MR3663611} 
for the effect of the adsorption term, 
and~\cite{WSCauchy} for the addition of the gravitational force term.
In particular, when the adsorption effect is included, the $c$
family described below would no longer be linearly degenerate, 
while adding the gravitational force term, the $s$ waves described below could have
negative speed. Both effects would disrupt the carefully designed wave
front tracking algorithm we use to prove the main theorem.

The conserved quantities and their fluxes are given by, respectively
\begin{displaymath}
  \mathbf{G}\left(s ,c ,k \right)
  =
  \begin{pmatrix}
    s \\
    c s \\
    k 
  \end{pmatrix},
\qquad
  \mathbf{F}\left(s ,c ,k \right)
  =
  \begin{pmatrix}
    f\left(s ,c ,k \right)\\
    c  f\left(s ,c ,k \right)\\
    0
  \end{pmatrix}.
\end{displaymath}
Denoting the three families as the $s$, $c$ and $k$ family, 
we have the following 3 eigenvalues as functions of the variables
$\left(\sigma,\gamma,\kappa\right)$ in the $\left(s,c,k\right)$ space.
\begin{displaymath}
  \lambda_s= \partial_{\sigma}f\left(\sigma,\gamma,\kappa\right), \qquad 
  \lambda_c=\frac{f\left(\sigma,\gamma,\kappa\right)}{\sigma},\qquad
  \lambda_k=0,
\end{displaymath}
and the three corresponding right eigenvectors (in the
$\left(s,c,k\right)$ space): 
\begin{displaymath}
  r_s=\begin{pmatrix}1 \\ 0 \\ 0 \end{pmatrix},
  \qquad
  r_c= 
  \begin{pmatrix} -\partial_{\gamma}f\left(\sigma,\gamma,\kappa\right)
 \\ 
 \partial_{\sigma}f\left(\sigma,\gamma,\kappa\right)-
 \frac{f\left(\sigma,\gamma,\kappa\right)}{\sigma} \\  0  \end{pmatrix}, 
  \qquad 
  r_k=\begin{pmatrix}
    -\partial_{\kappa}f\left(\sigma,\gamma,\kappa\right)\\ 0 
    \\ \partial_{\sigma}f\left(\sigma,\gamma,\kappa\right)\end{pmatrix}.
\end{displaymath}

A straight computation shows that both the $c$ and $k$ families are linearly degenerate.
Furthermore, there exist regions in the domain such that $\lambda_s=\lambda_c$ 
or $\lambda_s=\lambda_c=\lambda_k$, where the system is parabolic degenerate.

The flux function $f\left(\sigma,\gamma,\kappa\right)$
has the following properties.  
For any given $(\gamma,\kappa)$, the mapping $\sigma\mapsto
f\left(\sigma,\gamma,\kappa\right)$ is the well-known S-shaped 
Buckley-Leverett function~\cite{BL}
with a single inflection point, see Fig.~\ref{fig:fg}. 
To be specific, we have 
\begin{displaymath}
  f(\sigma,\gamma,\kappa)\in[0,1], \qquad \partial_{\sigma}f(\sigma,\gamma,\kappa)\ge0, \qquad \mbox{for all } (\sigma,\gamma,\kappa),
\end{displaymath}
and, for all $(\gamma,\kappa)$, 
\begin{equation}\label{fconds}
  \begin{split}
    &f(0,\gamma,\kappa)=0, \qquad \quad ~f(1,\gamma,\kappa)=1, \\
&    \partial_{\sigma} f(0,\gamma,\kappa)=0, \qquad
   ~ \partial_{\sigma}f(1,\gamma,\kappa) =0,\\
    &\partial_{\sigma\sigma}f(0,\gamma,\kappa) >0, \qquad
    \partial_{\sigma\sigma}f(1,\gamma,\kappa) <0.
  \end{split}
\end{equation}
Remark that conditions~\eqref{fconds} guarantee that the eigenvalues
and the eigenvectors written above are well defined (can be extended)
when $\sigma=0$.
For each given $(\gamma,\kappa)$, there exists a unique
$\sigma^*\left(\gamma,\kappa\right)\in\left]0,1\right[$ such that 
\[ \partial_{\sigma\sigma}
  f(\sigma^*\left(\gamma,\kappa\right),\gamma,\kappa) =0.\]

A detailed analysis of the wave properties for this system is carried
out in~\cite{WS3x3}, with the following highlights:
\begin{itemize}
\item $k$~waves are the slowest with speed $0$. Both  $f$ and $c$
  are continuous across any $k$~wave;
\item $c$~waves travel with non negative speed. 
Both  $\frac{f}{s}$ and $k$ are continuous across any $c$ wave;
\item $s$~waves  travel with positive speed.  Both $c$ and $k$ are
  continuous 
  across any $s$ wave.
\end{itemize}

In \cite{WS3x3}, the global Riemann solver is constructed. Here we give a brief summary.
Let $(s_l,c_l , k_l)$ and $(s_r,c_r,k_r)$ be the left and right state of a Riemann problem, respectively. In general, the solution of the Riemann problem consists of a $k$ wave, a $c$ wave
and possibly some $s$ waves. They can be constructed as follows.
\begin{itemize}
\item Let $(s_m,c_l, k_r)$
  denote the right state of the $k$ wave. The value  $s_m$
is uniquely determined by the condition 
$$ f(s_m,c_l,k_r) = f(s_l,c_l,k_l).$$

\item For the remaining waves, we have $k\equiv k_r$ throughout. 
We then solve the Riemann problem for the $2\times 2$ sub-system 
\begin{equation}\label{eq2}
\partial_{t}s + \partial_{x}f(s,c,k_r) =0, \qquad \partial_{t}(cs) + \partial_{x}(cf(s,c,k_r)) =0
\end{equation}
with Riemann data 
$(s_m,c_l)$ and $(s_r,c_r)$ as the left and right states.
The solution consists of waves with non-negative speed.
\end{itemize}

We now give a precise definition of weak solution to the
Cauchy problem~\eqref{eq1}--\eqref{eq:init_data} and state the main theorem.

\begin{definition}
  \label{def:main}
  The vector-valued function
  $\left(s,c,k\right)\in\mathbf{L}^{\infty}\left(
    [0,+\infty)\times \mathbb{R},[0,1]^{3}\right)$ is a solution to the Cauchy
  problem~\eqref{eq1}--\eqref{eq:init_data} if for any
  $\phi\in\textbf{C}^{1}_{c}\left(
    [0,+\infty)\times\mathbb{R},\mathbb{R}\right)$ the following equalities hold
  \begin{align}
    \nonumber
    \int_{\Omega} \left[s\partial_{t}\phi+f(s,c,k)\partial_{x}\phi\right](t,x)\;
    dtdx
    +\int_{\mathbb{R}}\bar s (x)\phi\left(0,x\right)\; dx =0,\\\nonumber
    \int_{\Omega} \left[cs\partial_{t}\phi+cf(s,c,k)\partial_{x}\phi\right](t,x)\;
    dtdx
    +\int_{\mathbb{R}}\bar c\left(x\right)\bar s
    (x)\phi\left(0,x\right)\; dx=0,\\
    \nonumber
    k\left(t,x\right)=\bar k(x), \quad \forall (t,x)\in \Omega,
  \end{align}
  where $\Omega=\left]0,+\infty\right[\times \mathbb{R}$.
\end{definition}

\begin{theorem}
  \label{thm:main}
  If the initial data $\left(\bar s, \bar c,\bar k\right)$ satisfy
  \[
  \bar s\in\mathbf{L}^{\infty}\left(\mathbb{R},[0,1]\right), \qquad 
  \bar c\in\mathbf{BV}\left(\mathbb{R},[0,1]\right), \qquad 
  \bar k\in\mathbf{BV}\left(\mathbb{R},[0,1]\right),
  \]
  then there exists a solution to the Cauchy
  problem~\eqref{eq1}--\eqref{eq:init_data} in the sense of
  Definition~\ref{def:main}.
\end{theorem}

We emphasize the fact that $s=0$
is not excluded in our
theorem since we do not make use of Lagrangian coordinates which would
have required $s>0$. 
Indeed, under the hypotheses $s\ge s^{*}>0$, $k(t,x)=\text{const.}$,
system~\eqref{eq1} is equivalent to its
Lagrangian formulation~\cite{Wagner}:
\begin{equation}
  \label{eq:lagr}
  \begin{cases}
    \partial_{t}\left(\frac{1}{s}\right)-\partial_{y}
    \left(\frac{f\left(s,c,k\right)}{s}\right)
    &=0,\\
    \partial_{t}c &= 0,\\
    k &= \text{const.},
  \end{cases}
\end{equation}
where $y$ is the Lagrangian coordinate satisfying $\partial_{x}y=s$,
$\partial_{t}y = -f(s,c,k)$.
Therefore, under some additional hypotheses, the result in~\cite{BGS},
which holds for \emph{scalar} equations since it is based on the
maximum principle for Hamilton--Jacobi equation,
could be used to prove the existence of a unique vanishing viscosity solution to the
first equation in~\eqref{eq:lagr} and hence a (in some sense) unique
solution to system~\eqref{eq:lagr}. 
However, since we consider the case where $s$ can become 0, the
analysis in~\cite{BGS} cannot be applied. 
Instead, we need to solve the original
non triangular $2\times 2$ system in Eulerian coordinates.
Furthermore, since we consider rough permeability function $k$, 
the corresponding system in the Lagrangian coordinate is no longer triangular, 
\begin{displaymath}
  \begin{cases}
    \partial_{t}\left(\frac{1}{s}\right)-\partial_{y}
    \left(\frac{f\left(s,c,k\right)}{s}\right)
    &=0,\\
    \partial_{t}c &= 0,\\
    \partial_{t}\left(\frac{k}{s}\right)-\partial_{y}\left(\frac{kf\left(s,c,k\right)}
      {s}\right)&=0.
  \end{cases}
\end{displaymath}
Remark that the Lagrangian coordinates introduced
in~\cite[Section~6]{BGS} for \eqref{eq1} 
make the system triangular, but still require $s\ge s^{*}>0$ and moreover
they mix time and space, therefore the Cauchy problem for~\eqref{eq1}
in Eulerian
coordinates is not equivalent to the Cauchy problem in the
coordinates introduced in~\cite[Section~6]{BGS}.

In this paper, the proof for the existence of solution
is carried out by showing that wave front tracking approximate
solutions are compact by a compensated compactness argument (see for
instance~\cite{KarRasTad} for an application of compensated
compactness to a $2\times2$ bi--dimensional related model).

The remaining of the paper is organized as follows. 
In Section~\ref{sec:FT} the wave front tracking approximate solutions are
constructed. In Section~\ref{sec:entropy}  we prove the necessary entropy
estimates. Finally in Section~\ref{sec:conv} the
compensated compactness argument is carried out to prove
Theorem~\ref{thm:main}.

\section{Front Tracking Approximations}
\label{sec:FT}
In this section we modify the algorithm
constructed in \cite{CocliteRisebro} and \cite{WSCauchy} to adapt it to
system~\eqref{eq1}. We define the functions (see Figure~\ref{fig:fg}):
\begin{equation}
  \label{eq:sig_trans}
  g\left(\sigma,\gamma,\kappa\right)=\frac{f\left(\sigma,\gamma,\kappa\right)}{\sigma},
  \qquad
  P\left(\sigma,\gamma,\kappa\right)=\int_{0}^{\sigma}\left|\partial_{\xi}g
    \left(\xi,\gamma,\kappa\right)\right|d\xi.
\end{equation}
Since at $\sigma=0$ both $f$ and its derivative $\partial_\sigma f$ vanish, we define
$g(0,\gamma,\kappa)=0$. Hypotheses~\eqref{fconds} imply that
$\partial_{\sigma}g\left(0,\gamma,\kappa\right)>0$ and that the function
$\sigma\mapsto
g\left(\sigma,\gamma,\kappa\right)$ has one single maximum point somewhere
between the single inflexion point of $f$ and the point $\sigma=1$.
The function $P$ is continuous with respect to its three variables and
strictly increasing and invertible with respect to the variable
$\sigma$. 
\begin{figure}
  \begin{center}
    \begin{tikzpicture}[line cap=round,line join=round,x=4cm,y=3cm,
      declare function={ f(\s,\c,\k) = \s*\s / (\s*\s + (1 - \k/4*(\c
        - 0.5)*(\c - 0.5))*(1-\s)*(1-\s)); g(\s,\c,\k) = \s / (\s*\s +
        (1 - \k/4*(\c - 0.5)*(\c - 0.5))*(1-\s)*(1-\s)); P(\s,\c,\k) =
        \s / (\s*\s + (1 - \k/4*(\c - 0.5)*(\c - 0.5))*(1-\s)*(1-\s));
      },]
      \draw [line width=0.5pt,->] (-0.1,0)-- (1.2,0); \draw [line
      width=0.5pt,->] (0.0,-0.1)-- (0,1.3); \draw (1.2,0) node
      [label=below:{$\sigma$}] {}; \draw (1.,1) node
      [label=right:{$\left(1,1\right)$}] {}; \draw (1.,0) node
      [label=below:{$\left(1,0\right)$}] {}; \draw (0.,0) node
      [label=below left:{$\left(0,0\right)$}] {}; \draw (1.,1) node
      [label=below left:{$\textcolor{red}{f}$}] {}; \draw (1.,1.1)
      node [label=above left:{$\textcolor{blue}{g}$}] {};
      \draw[color=red,line width=2] plot[domain=0:1,samples=100]
      (\x,{f(\x,1,1)}); \draw[color=blue,line width=2]
      plot[domain=0:1,samples=100] (\x,{g(\x,1,1)}); \draw (0,0) node
      [fill,circle,inner sep=2pt] {}; \draw (1,0) node
      [fill,circle,inner sep=2pt] {}; \draw (1,1) node
      [fill,circle,inner sep=2pt] {};
    \end{tikzpicture}
  \end{center}
  \caption{Diagrams of the flux $\sigma\mapsto
    f\left(\sigma,\gamma,\kappa\right)$ and of the function
    $\sigma\mapsto g\left(\sigma,\gamma,\kappa\right)=
  f\left(\sigma,\gamma,\kappa\right)/\sigma$ for fixed values of
  $\gamma$ and $\kappa$.}
\label{fig:fg}
\end{figure}
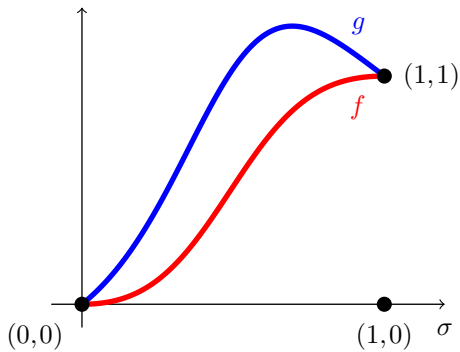
Fixing initial data $\left(\bar s, \bar c, \bar k\right)$ satisfying the
hypotheses of Theorem~\ref{thm:main} and
fixing an approximation parameter $\varepsilon>0$, we can construct piecewise
constant approximate initial data $\left(\bar s^{\varepsilon},\bar
  c^{\varepsilon},
  \bar k^{\varepsilon}\right)$ with values in $\left[0,1\right]^{3}$ such that:
\begin{align}
  \nonumber
    \left\|\bar
      k^{\varepsilon} - \bar
  k\right\|_{\mathbf{L}^{\infty}}<\varepsilon,&&
  \tv \bar k^{\varepsilon} \le \tv \bar k,
    \\
  \nonumber
    \left\|\bar
      c^{\varepsilon} - \bar
       c\right\|_{\mathbf{L}^{\infty}}<\varepsilon,&&\tv \bar c^{\varepsilon}\le \tv \bar c ,
    \\
  \left\|\bar s^{\varepsilon}-\bar
  s\right\|_{\mathbf{L}^{1}\left(\left(-\frac{1}{\varepsilon},\frac{1}{\varepsilon}\right),
  \mathbb{R}\right)}\le \varepsilon,&&
\end{align}
Let $\bar x_{1},\ldots,\bar x_{N}$ be the set of points in which
$\bar
k^{\varepsilon}$ has
jumps such that
\begin{displaymath}
  \bar k^{\varepsilon}\left(x\right)=
  k_{0}\chi_{]-\infty,\bar x_{1}]} +
  \sum_{i=1}^{N-1}k_{i}\chi_{]\bar
    x_{i},\bar x_{i+1}]}(x)
  +k_{N}\chi_{]\bar x_{N},+\infty[},
\end{displaymath}
and let $\bar y_{1},\ldots,\bar y_{M}$ be the set of points in which
$\bar c^{\varepsilon}$ has
jumps such that
\begin{displaymath}
  \bar c^{\varepsilon}\left(x\right)=
  c_{0}\chi_{]-\infty,\bar y_{1}]} + \sum_{j=1}^{M-1}c_{j}\chi_{]\bar
    y_{j},\bar y_{j+1}]}(x)
  +c_{M}\chi_{]\bar y_{M},+\infty[}.
\end{displaymath}
Without loss of generality, we can suppose that no $\bar y_{j}$ coincides with any
$\bar x_{i}$.
Define the constant
\begin{displaymath}
L=\left\lceil
  \frac{1}{\varepsilon}
  \sup_{\gamma,\kappa}\left\|f\left(\cdot,\gamma,\kappa\right)\right\|_{\mathbf{C}^{2}}
\right\rceil\cdot\left(N+M\right),
\end{displaymath}
where $\lceil \alpha \rceil$ denotes the least integer greater than or equal
to the real number $\alpha$. In the following we denote by $\land$ the
logical operator \textbf{and}. 
We consider the following finite sets of possible values for the
function $g$:
\begin{align*}
    \mathcal{G}_{0}^{1}&=\left\{g\mid g=g\left(\bar s^{\varepsilon}(x),\bar
        c^{\varepsilon}(x), \bar k^{\varepsilon}(x)\right),\quad
      x\in\mathbb{R}\right\},\\
    \mathcal{G}_{0}^{2}&=\left\{g\mid
                           g=g\left(\frac{\ell}{L},c_{j},k_{i}\right),\quad
                           i=0,\ldots,N,\;
                           j=0,\ldots,M,\;
                           \ell=0,\ldots,L\right\},\\
   \mathcal{G}_{0}^{3}&=\left\{g\mid g=\max_{0\le \sigma \le
        1}g\left(\sigma,c_{j},k_{i}\right),\quad
      i=0,\ldots,N,\;j=0,\ldots,M\right\},\\
    \mathcal{G}_{0}^{4}&=\Big\{g\mid
      g=g\left(\sigma,c_{j},k_{i}\right)=g\left(\sigma,c_{j^{*}},k_{i^{*}}\right)
      \land g_{s}\left(\sigma,c_{j},k_{i}\right)\cdot
                         g_{s}\left(\sigma,c_{j^{*}},k_{i^{*}}\right)<
                     0,\\
                       &\qquad\qquad \text{
                         for some }
                         i,i^{*}=0,\ldots,N,\;
                           j,j^{*}=0,\ldots,M,\;\sigma\in[0,1]
\Big\},\\
    \mathcal{G}_{0}&=\mathcal{G}_{0}^{1}
                       \cup \mathcal{G}_{0}^{2}\cup\mathcal{G}_{0}^{3}
                       \cup\mathcal{G}_{0}^{4}.
\end{align*}
\begin{remark}~
  \begin{itemize}
  \item
    The set $\mathcal{G}_{0}^{1}$ includes all the possible initial values for $g$;
  \item
    The set $\mathcal{G}_{0}^{2}$ includes a sufficiently fine grid for
    $g$ in order that any $s$ grid that contains all the counter images of
    $\mathcal{G}_{0}^{2}$ through $g\left(\cdot,c_{j},k_{i}\right)$,
    for any fixed $j,i$, is
    finer than $\frac{1}{L}$;
  \item
    The set $\mathcal{G}_{0}^{3}$ includes all the possible maxima of
    $g\left(\cdot,c_{j},k_{i}\right)$ for any $j,i$;
  \item
    The set $\mathcal{G}_{0}^{4}$ includes all the possible
    values of $g$ where two graphs of functions of type 
    $g\left(\cdot,c_{j},k_{i}\right)$ intersect with derivatives of different
    sign. Because of the shape of $g$, this
    set too is finite.  
  \end{itemize}
\end{remark}

We start the front tracking algorithm from the region $x<\bar
x_{1}$. For this purpose
we define the allowed values for $s$ in
that region:
\begin{displaymath}
  \mathcal{S}_{0,j}=\left\{\sigma\mid g\left(\sigma,c_{j},k_{0}\right)\in
    \mathcal{G}_{0}\right\}.
\end{displaymath}
We  call $f^{0,j}(\sigma)$ the
linear interpolation of the map
$\sigma\mapsto f\left(\sigma,c_{j},k_{0}\right)$ according to the
points in $\mathcal{S}_{0,j}$.
Observe that, since we have included $\mathcal{G}_{0}^{2}$ in
$\mathcal{G}_{0}$, the set $\left\{\frac{\ell}{L}\right\}_{\ell=0}^{L}$ is
included in $\mathcal{S}_{0,j}$ for every $j$, hence we have the
uniform estimate
\begin{displaymath}
  \begin{cases}
    \left|f\left(\sigma,c_{j},k_{0}\right)-f^{0,j}(\sigma)\right|\le
    \varepsilon,\\
    \left|\partial_{\sigma}f\left(\sigma,c_{j},k_{0}\right)-\partial_{\sigma}f^{0,j}(\sigma)\right|\le
    \frac{\varepsilon}{N+M} ,   
  \end{cases}
  \quad\text{ for all }\sigma\in\left[0,1\right],\ j=0,\ldots,M.
\end{displaymath}
\begin{figure}
\begin{tikzpicture}[line cap=round,line join=round,>=triangle 45,x=0.6cm,y=0.6cm]
\draw (-9,4) node [label=below:{$\Omega_{0,0}$}] {};
\draw (-7,2) node [label=below:{$\Omega_{0,1}$}] {};
\draw (-5.,1.5) node [label=below:{$\Omega_{1,2}$}] {};
\draw (-5.,8) node [label=below:{$\Omega_{1,4}$}] {};
\draw (1.,4) node [label=below:{$\Omega_{3,3}$}] {};
\draw (4.5,6) node [label=below:{$\Omega_{4,3}$}] {};
\draw [line width=2.pt,color=blue] (-6.,0.)-- (-6.,9.);
\draw (-6,0) node [label=below:{$\bar x_{1}$}] {};
\draw [line width=2.pt,color=blue] (-3.,0.)-- (-3.,9.);
\draw (-3,0) node [label=below:{$\bar x_{2}$}] {};
\draw [line width=2.pt,color=blue] (-1.,9.)-- (-1.,0.);
\draw (-1,0) node [label=below:{$\bar x_{3}$}] {};
\draw [line width=2.pt,color=blue] (3.,9.)-- (3.,0.);
\draw (3.,0) node [label=below:{$\bar x_{4}$}] {};
\draw [line width=2.pt,color=blue] (6.,9.)-- (6.,0.);
\draw (6.,0) node [label=below:{$\bar x_{5}$}] {};
\draw [line width=2.pt,color=green] (-7.34,0.)-- (-6.,1.52)-- (-3.,4.)-- (-1.,6.)-- (3.,7.78)-- (6.06,9.72);
\draw (-7.34,0) node [label=below:{$\bar y_{2}$}] {};
\draw [line width=2.pt,color=green] (-5.,0.)-- (-3.,2.26)-- (-1.,3.2)-- (1.26,4.94)-- (3.,5.62)-- (6.,7.26)-- (8.72,8.86);
\draw (-5,0) node [label=below:{$\bar y_{3}$}] {};
\draw [line width=2.pt,color=green] (-9.74,0.)-- (-6.,3.28)-- (-3.,6.)-- (-1.,7.7)-- (3.,9.3);
\draw (-9.74,0) node [label=below:{$\bar y_{1}$}] {};
\draw [line width=2.pt,color=green] (-2.22,0.)-- (-1.,1.)-- (3.,3.82)-- (6.,5.62)-- (8.84,7.32);
\draw (-2.22,0) node [label=below:{$\bar y_{4}$}] {};
\draw [line width=2.pt,color=green] (0.86,0.)-- (3.,2.)-- (4.54,2.8)-- (6.,4.26)-- (8.86,5.36);
\draw (0.86,0) node [label=below:{$\bar y_{5}$}] {};
\draw [line width=2.pt,color=green] (2.44,0.)-- (3.,0.4)-- (6.,2.5)-- (9.,4.);
\draw [line width=2.pt,color=green] (4.64,0.)-- (6.,1.)-- (9.,2.32);
\draw (4.64,0) node [label=below:{$\bar y_{7}$}] {};
\draw [->,line width=2.pt] (-10.,0.) -- (9,0.);
\draw [line width=1.pt,color=red] (-10,0.)-- (-4.92,6.14)-- (-3.,7.12)-- (-1.,8.12)-- (1.4206896551724135,8.668275862068967)-- (2.88,9.86);
\draw [line width=1.pt,color=red] (-8.5,0.)-- (-7.693201648751313,1.7950530994908267)-- (-6.,3.96)-- (-4.92,6.14)-- (-3.,7.12)-- (-1.,8.12)-- (1.4206896551724135,8.668275862068967)-- (2.88,9.86);
\draw [line width=1.pt,color=red] (-7.693201648751313,1.7950530994908267)-- (-6.,2.4)-- (-4.000844862181856,3.1726349139296657)-- (-3.,4.78)-- (-1.,6.82)-- (0.3034482758620687,8.221379310344828)-- (0.7498883001824026,8.516342221836187);
\draw [line width=1.pt,color=red] (-4.000844862181856,3.1726349139296657)-- (-3.,3.56)-- (-1.,4.76)-- (1.8,6.24)-- (3.,6.92)-- (6.,8.84)-- (7.56,9.78);
\draw [line width=1.pt,color=red] (-4.46,0.)-- (-3.,1.)-- (-1.,2.22)-- (1.26,4.94)-- (3.,6.06)-- (4.27128263337117,7.733620885357547);
\draw [line width=1.pt,color=red] (1.26,4.94)-- (3.,5.2)-- (5.096632904462085,6.76615932110594)-- (6.,7.76)-- (7.22,8.82);
\draw [line width=1.pt,color=red] (5.096632904462085,6.76615932110594)-- (6.,6.96)-- (7.62,7.72);
\draw [line width=1.pt,color=red] (1.86,0.)-- (3.,1.1)-- (4.54,2.8)-- (6.,3.62)-- (8.78,4.72);
\draw [line width=1.pt,color=red] (4.54,2.8)-- (6.,4.88)-- (8.76,5.74);
\draw [line width=1.pt,color=red] (-1.52,0.)-- (-1.,0.56)-- (1.299293598971293,2.6210019872747616)-- (3.,4.36)-- (3.9326190698602153,5.896655120786378);
\end{tikzpicture}
\caption{Wave front tracking pattern. $k$ waves in blue, $c$ waves
  in green and $s$ waves in red}
\end{figure}
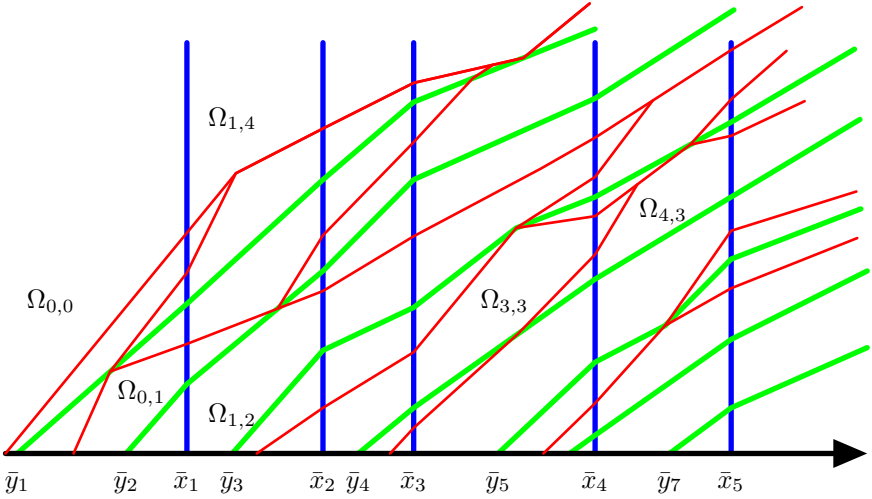

We solve all the Riemann problems in $x<\bar x_{1}$ at $t=0$ in the
following way. Let $\bar x \in\left]\bar y_{j},\bar y_{j+1}\right[$
($\bar y_{0}=-\infty$) be a jump in $\bar s^{\varepsilon}$. Here we take the
entropic solution to the Riemann problem
\begin{displaymath}
  \partial_{t}s + \partial_{x}f^{0,j}\left(s\right)=0,\qquad
  s\left(0,x\right)=
  \begin{cases}
    \bar s^{\varepsilon}\left(\bar x-\right)&\text{ for }x<\bar x,\\
    \bar s^{\varepsilon}\left(\bar x+\right)&\text{ for }x>\bar x.
  \end{cases}
\end{displaymath}
Since $f^{0,j}$ is piecewise linear, the solution to the Riemann
problem is piecewise constant and takes values in the set
$\mathcal{S}_{0,j}$ of the kink points of 
$f^{0,j}$, moreover the entropy condition
in~\cite[Theorem~4.4]{Bbook} is satisfied.  The same Riemann solver is
used whenever at $t>0$ two $s$~waves interact in some region
$\Omega_{0,j}$ defined below.

At the
points $\bar y_{j}$, we solve the Riemann problem
according to the minimum jump
condition
described in~\cite{WSCauchy} (see
also~\cite{GimseRisebro}) that we briefly outline (see Fig~\ref{fig:fgRP}).
\begin{figure}
  \begin{center}
    \begin{tikzpicture}[line cap=round,line join=round,x=3cm,y=3cm,
      declare function={
        g(\s,\c,\k) = \s / (\s*\s + (1 - \k/4*(\c - 0.5)*(\c -
        0.5))*(1-\s)*(1-\s)); 
      },]
      \draw [line width=0.5pt,->] (-0.1,0)-- (1.1,0);
      \draw [line width=0.5pt,->] (0.0,-0.1)-- (0,1.3);
      \draw (1.1,0) node [label=below:{$\sigma$}] {};
      \draw[color=black,line width=1]
      plot[domain=0:1,samples=100] (\x,{g(\x,1,1)});
      \draw (0,0) node [fill,circle,inner sep=1pt] {};
      \draw (0.3,0) node [label=below:{$s^{L}=s^{-}$}] {};
      \draw (0.3,0) node [fill,circle,inner sep=1pt] {};
      \draw [color=black,dashed,line width=0.5] (0.3,0)--(0.3,{g(0.3,1,1)});
      \draw [color=blue,line width=2] (0.0,{g(0.3,1,1)})--(1,{g(0.3,1,1)});
      \draw (1,1) node [fill,circle,inner sep=1pt] {};
      \draw (0.,{g(0.3,1,1)}) node [label=left:{$\textcolor{green}{\gamma}$}] {};
    \end{tikzpicture}
    \begin{tikzpicture}[line cap=round,line join=round,x=3cm,y=3cm,
      declare function={
        g(\s,\c,\k) = \s / (\s*\s + (1 - \k/4*(\c - 0.5)*(\c -
        0.5))*(1-\s)*(1-\s)); 
      },]
      \draw [line width=0.5pt,->] (-0.1,0)-- (1.1,0);
      \draw [line width=0.5pt,->] (0.0,-0.1)-- (0,1.3);
      \draw (1.1,0) node [label=below:{$\sigma$}] {};
      \draw[color=black,line width=1]
      plot[domain=0:1,samples=100] (\x,{g(\x,1,1)});
      \draw (0,0) node [fill,circle,inner sep=1pt] {};
      \draw (0.6,0) node [label=below:{$s^{L}$}] {};
      \draw (0.6,0) node [fill,circle,inner sep=1pt] {};
      \draw [color=black,dashed,line width=0.5] (0.6,0)--(0.6,{g(0.6,1,1)});
      \draw [color=blue,line width=2] (0.0,{g(0.6,1,1)})--(0.8,{g(0.6,1,1)});
      \draw[color=blue,line width=2]
      plot[domain=0.8:1,samples=100] (\x,{g(\x,1,1)});
      \draw (1,1) node [fill,circle,inner sep=1pt] {};
      \draw (0.9,0) node [fill,circle,inner sep=1pt] {};
      \draw [color=black,dashed,line width=0.5] (0.9,0)--(0.9,{g(0.9,1,1)});
      \draw (0.9,0) node [label=below:{$s^{-}$}] {};
      \draw [color=green,line width=0.5] (0.0,{g(0.9,1,1)})--(1.0,{g(0.9,1,1)});
      \draw (0.,{g(0.9,1,1)}) node [label=left:{$\textcolor{green}{\gamma}$}] {};
    \end{tikzpicture}
    \begin{tikzpicture}[line cap=round,line join=round,x=3cm,y=3cm,
      declare function={
        g(\s,\c,\k) = \s / (\s*\s + (1 - \k/4*(\c - 0.5)*(\c -
        0.5))*(1-\s)*(1-\s)); 
      },]
      \draw [line width=0.5pt,->] (-0.1,0)-- (1.1,0);
      \draw [line width=0.5pt,->] (0.0,-0.1)-- (0,1.3);
      \draw (1.1,0) node [label=below:{$\sigma$}] {};
      \draw[color=black,line width=1]
      plot[domain=0:1,samples=100] (\x,{g(\x,1,1)});
      \draw (0,0) node [fill,circle,inner sep=1pt] {};
      \draw (0.7,0) node [label=below:{$s^{-}$}] {};
      \draw (0.7,0) node [fill,circle,inner sep=1pt] {};
      \draw [color=black,dashed,line width=0.5] (0.7,0)--(0.7,{g(0.7,1,1)});
      \draw [color=blue,line width=2] (0.0,{g(0.7,1,1)})--(0.7,{g(0.7,1,1)});
      \draw[color=blue,line width=2]
      plot[domain=0.7:1,samples=100] (\x,{g(\x,1,1)});
      \draw (1,1) node [fill,circle,inner sep=1pt] {};
      \draw (0.9,0) node [fill,circle,inner sep=1pt] {};
      \draw [color=black,dashed,line width=0.5] (0.9,0)--(0.9,{g(0.9,1,1)});
      \draw (0.9,0) node [label=below:{$s^{L}$}] {};
      \draw [color=green,line width=0.5] (0.0,{g(0.7,1,1)})--(1.0,{g(0.7,1,1)});
      \draw (0.,{g(0.7,1,1)}) node [label=left:{$\textcolor{green}{\gamma}$}] {};
    \end{tikzpicture}
    \begin{tikzpicture}[line cap=round,line join=round,x=3cm,y=3cm,
      declare function={
        g(\s,\c,\k) = \s / (\s*\s + (1 - \k/4*(\c - 0.5)*(\c -
        0.5))*(1-\s)*(1-\s)); 
      },]
      \draw [line width=0.5pt,->] (-0.1,0)-- (1.1,0);
      \draw [line width=0.5pt,->] (0.0,-0.1)-- (0,1.3);
      \draw (1.1,0) node [label=below:{$\sigma$}] {};
      \draw[color=black,line width=1]
      plot[domain=0:1,samples=100] (\x,{g(\x,1,1)});
      \draw[color=red,line width=2]
      plot[domain=0.:0.7,samples=100] (\x,{g(\x,1,1)});
      \draw (0,0) node [fill,circle,inner sep=1pt] {};
      \draw (0.3,0) node [label=below:{$s^{R}\ $}] {};
      \draw (0.3,0) node [fill,circle,inner sep=1pt] {};
      \draw (0.4,0) node [label=below:{$\ s^{+}$}] {};
      \draw (0.4,0) node [fill,circle,inner sep=1pt] {};
      \draw [color=black,dashed,line width=0.5] (0.3,0)--(0.3,{g(0.3,1,1)});
      \draw [color=black,dashed,line width=0.5] (0.4,0)--(0.4,{g(0.4,1,1)});
      \draw [color=red,line width=2] (0.7,{g(0.7,1,1)})--(1,{g(0.7,1,1)});
      \draw [color=green,line width=0.5] (0.0,{g(0.4,1,1)})--(1.0,{g(0.4,1,1)});
      \draw (1,1) node [fill,circle,inner sep=1pt] {};
      \draw (0.,{g(0.4,1,1)}) node [label=left:{$\textcolor{green}{\gamma}$}] {};
    \end{tikzpicture}
    \begin{tikzpicture}[line cap=round,line join=round,x=3cm,y=3cm,
      declare function={
        g(\s,\c,\k) = \s / (\s*\s + (1 - \k/4*(\c - 0.5)*(\c -
        0.5))*(1-\s)*(1-\s)); 
      },]
      \draw [line width=0.5pt,->] (-0.1,0)-- (1.1,0);
      \draw [line width=0.5pt,->] (0.0,-0.1)-- (0,1.3);
      \draw (1.1,0) node [label=below:{$\sigma$}] {};
      \draw[color=black,line width=1]
      plot[domain=0:1,samples=100] (\x,{g(\x,1,1)});
      \draw (0,0) node [fill,circle,inner sep=1pt] {};
      \draw (0.4,0) node [label=below:{$s^{+}$}] {};
      \draw (0.4,0) node [fill,circle,inner sep=1pt] {};
      \draw [color=black,dashed,line width=0.5] (0.4,0)--(0.4,{g(0.4,1,1)});
      \draw[color=red,line width=2]
      plot[domain=0.0:0.54,samples=100] (\x,{g(\x,1,1)});
      \draw (1,1) node [fill,circle,inner sep=1pt] {};
      \draw (0.9,0) node [fill,circle,inner sep=1pt] {};
      \draw [color=black,dashed,line width=0.5] (0.9,0)--(0.9,{g(0.9,1,1)});
      \draw (0.9,0) node [label=below:{$s^{R}$}] {};
      \draw [color=red,line width=2] (0.55,{g(0.9,1,1)})--(1.0,{g(0.9,1,1)});
      \draw (0.,{g(0.4,1,1)}) node [label=left:{$\textcolor{green}{\gamma}$}] {};
      \draw [color=green,line width=0.5] (0.0,{g(0.4,1,1)})--(1.0,{g(0.4,1,1)});
    \end{tikzpicture}
  \end{center}
  \caption{The graphs of $G^{L}$ and $G^{R}$ are drawn respectively in
    blue and red. For each graph, a possible transition level $\gamma$ (the level at
    which given $G^{L}$ and $G^{R}$ intersect) is drawn in green.}
\label{fig:fgRP}
\end{figure}
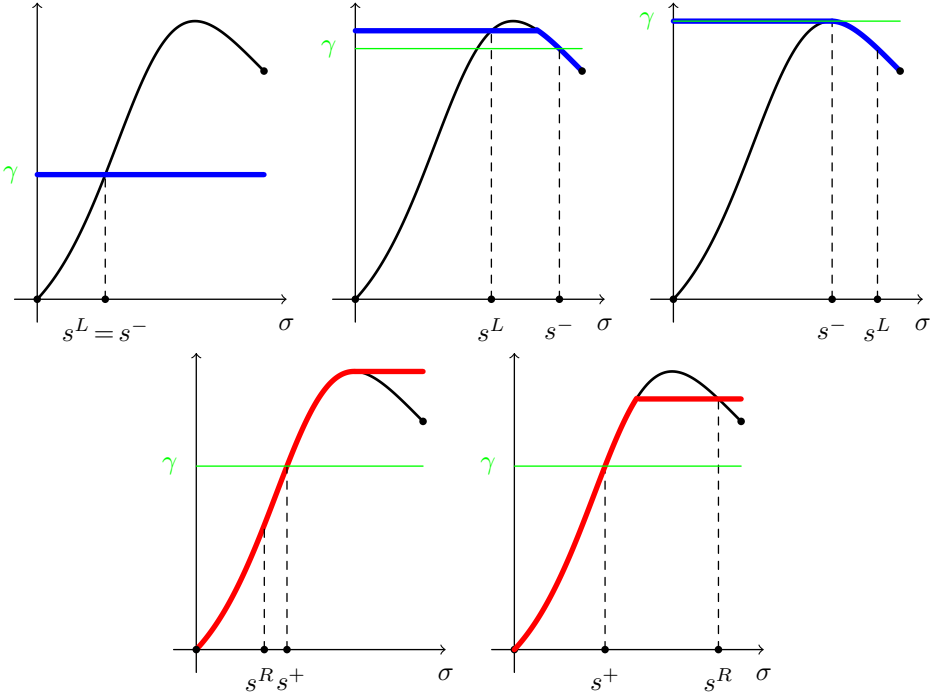
Define
\begin{displaymath}
  s^{L}=\bar s^{\varepsilon}\left(\bar y_{j}-\right),\quad
  c^{L}=\bar c^{\varepsilon}\left(\bar y_{j}-\right)=c_{j-1},\quad
  s^{R}=\bar s^{\varepsilon}\left(\bar y_{j}+\right),\quad
  c^{R}=\bar c^{\varepsilon}\left(\bar y_{j}+\right)=c_{j},
\end{displaymath}
and the two auxiliary monotone functions (the first one non increasing and
the second one non decreasing)
\begin{displaymath}
G^{L}\left(\sigma\right)=
\begin{cases}
  \max \left\{g\left(\varsigma,c^{L},k_{0}\right)\mid
    \varsigma\in\left[\sigma,s^{L}\right]\right\},&\text{ for
  }\sigma\le s^{L},\\[10pt]
  \min \left\{g\left(\varsigma,c^{L},k_{0}\right)\mid
    \varsigma\in\left[s^{L},\sigma\right]\right\},&\text{ for
  }\sigma\ge s^{L},
\end{cases}
\end{displaymath}
\begin{displaymath}
G^{R}\left(\sigma\right)=
\begin{cases}
  \min \left\{g\left(\varsigma,c^{R},k_{0}\right)\mid
    \varsigma\in\left[\sigma,s^{R}\right]\right\},&\text{ for
  }\sigma\le s^{R},\\[10pt]
  \max \left\{g\left(\varsigma,c^{R},k_{0}\right)\mid
    \varsigma\in\left[s^{R},\sigma\right]\right\},&\text{ for
  }\sigma\ge s^{R}.
\end{cases}
\end{displaymath}
Call $\gamma$ the unique level at which $G^{L}$ and
$G^{R}$ intersect. Because of the hypotheses on $g$, $\gamma$ is equal
to either $g\left(s^{L},c^{L},k_{0}\right)$,
$g\left(s^{R},c^{R},k_{0}\right)$, a maximum of  either
$g\left(\cdot,c^{L},k_{0}\right)$ or
$g\left(\cdot,c^{R},k_{0}\right)$, or a point in which these two
function intersect with derivatives having opposite sign.
In any case  $\gamma\in\mathcal{G}_{0}$ holds.
Define the closed intervals
\begin{displaymath}
  I^{L} =\left[G^{L}\right]^{-1}\left(\left\{\gamma\right\}\right),\qquad
  I^{R}=\left[G^{R}\right]^{-1}\left(\left\{\gamma\right\}\right).
\end{displaymath}
Finally call $s^{-}$ and $s^{+}$ respectively the unique projections of 
$s^{L}$ and $s^{R}$ on the closed strictly convex sets $I^{L}$ and $I^{R}$.
It is not difficult to show that
\begin{displaymath}
  \begin{array}{ll}
    \gamma=G^{L}\left(s^{-}\right)=g\left(s^{-},c^{L},k_{0}\right),\quad&s^{L},s^{-}\in \mathcal{S}_{0,j-1},\\
    \gamma=G^{R}\left(s^{+}\right)= g\left(s^{+},c^{R},k_{0}\right),
    \quad &s^{R},\;s^{+}\in \mathcal{S}_{0,j}.
  \end{array}
\end{displaymath}
Take any wave, with left and right
states $s_{l},s_{r}\in \mathcal{S}_{0,j-1}$, of the entropic
solution to the Riemann problem
\begin{equation}
  \label{eq:j-1}
  \partial_{t}s + \partial_{x}f^{0,j-1}\left(s\right)=0,\qquad
  s\left(0,x\right)=
  \begin{cases}
    s^{L}&\text{ for }x<\bar y_{j},\\
    s^{-}&\text{ for }x>\bar y_{j},
  \end{cases}
\end{equation}
and suppose $s^{L}< s^{-}$. Then,
$s^{L}\le s_{l}<s_{r}\le s^{-}$ and,
because of the entropy condition~\cite[Theorem~4.4]{Bbook}, 
its speed satisfies ($f^{0,j-1}$ coincides with
$f\left(\cdot,c^{L},k_{0}\right)$ on $\mathcal{S}_{0,j-1}$):
\begin{displaymath}
  \begin{split}
    \lambda_{s}&=
    \frac{f\left(s_{r},c^{L},k_{0}\right)-f\left(s_{l},c^{L},k_{0}\right)}{s_{r}-s_{l}}
    \le
    \frac{f\left(s^{-},c^{L},k_{0}\right)-f\left(s_{l},c^{L},k_{0}\right)}{s^{-}-s_{l}}\\
    &\quad = g\left(s^{-},c^{L},k_{0}\right) +
    s_{l}\frac{g\left(s^{-},c^{L},k_{0}\right)-g\left(s_{l},c^{L},k_{0}\right)}{s^{-}-s_{l}}\\
    &\quad = \lambda_{c} +
    s_{l}\frac{G^{L}\left(s^{-}\right)-g\left(s_{l},c^{L},k_{0}\right)}{s^{-}-s_{l}}\\
    &\quad \le \lambda_{c},
  \end{split}
\end{displaymath}
with
$\lambda_{c}=\gamma=
g\left(s^{-},c^{L},k_{0}\right)=g\left(s^{+},c^{R},k_{0}\right)$
and where we used the definition of
$G^{L}$. If instead $s^{-}<s^{L}$, then $s^{-}\le s_{r}<s_{l}\le
s^{L}$ and, as in the previous computation, we have 
\begin{displaymath}
    \lambda_{s}\le  \lambda_{c} +
    s_{l}\frac{G^{L}\left(s^{-}\right)-g\left(s_{l},c^{L},k_{0}\right)}{s^{-}-s_{l}}
    \le \lambda_{c},
\end{displaymath}
Therefore, in any case, the solution to the Riemann problem~\eqref{eq:j-1}
can be patched with a $c$ wave that travel with speed $\lambda_{c}$
and connect the left state $\left(s^{-},c^{L},k_{0}\right)$ to the
right state $\left(s^{+},c^{R},k_{0}\right)$. Similar computations
can be done at
the right of the $c$ wave so that
the complete solution includes a $c$ wave travelling with speed
$\lambda_{c}$, possibly together some entropic $s$ waves to its left
(solutions to $\partial_{t}s+\partial_{x}f^{0,j-1}\left(s\right)=0$) and some entropic
$s$ waves to its right (solutions to
$\partial_{t}s+\partial_{x}f^{0,j}\left(s\right)=0$). 
We also use this Riemann solver whenever, for $t>0$,
a $c$ wave interact with one or more $s$ waves.

We point out the following properties of this Riemann solver that will
be needed in the proof of the main theorem.
\begin{itemize}
\item
  The $c$ wave satisfies Rankine-Hugoniot
  \begin{displaymath}
    \begin{cases}
      f\left(s^{+},c^{R},k_{0}\right)-f\left(s^{-},c^{L},k_{0}\right)=\lambda_{c}
      \left(s^{+}-s^{-}\right),\\
      c^{R}f\left(s^{+},c^{R},k_{0}\right)-c^{L}f\left(s^{-},c^{L},k_{0}\right)
      =\lambda_{c}
      \left(c^{R}s^{+}-c^{L}s^{-}\right).\\
    \end{cases}
  \end{displaymath}
\item
  The $c$ wave is an ``admissible'' path as defined in~\cite{WSCauchy}
  and satisfies the following entropy condition:
  \begin{itemize}
  \item if $s^{-}<s^{+}$ there exists
    $s^{*}\in\left[s^{-},s^{+}\right]$ such that
    \begin{equation}
      \label{eq:entropy2}
      \begin{cases}
        g\left(\sigma,c^{L},k_{0}\right)\ge \lambda_{c}
        &\text{ for all }\sigma\in\left[s^{-},s^{*}\right],\\
        g\left(\sigma,c^{R},k_{0}\right)\ge \lambda_{c}
        &\text{ for all }\sigma\in\left[s^{*},s^{+}\right].
      \end{cases}
    \end{equation}
  \item if $s^{+}<s^{-}$ there exists
    $s^{*}\in\left[s^{+},s^{-}\right]$ such that
    \begin{equation}
      \label{eq:entropy1}
      \begin{cases}
        g\left(\sigma,c^{R},k_{0}\right)\le \lambda_{c}
        &\text{ for all }\sigma\in\left[s^{+},s^{*}\right],\\
        g\left(\sigma,c^{L},k_{0}\right)\le \lambda_{c}
        &\text{ for all }\sigma\in\left[s^{*},s^{-}\right].
      \end{cases}
    \end{equation}
  \end{itemize}
\end{itemize}

\noindent
Let $y_{1}\left(t\right),\ldots,y_{M}\left(t\right)$ denote all the
$c$ wave
fronts at the time $t$. Their initial positions are
the discontinuity points of $\bar c^{\varepsilon}$. 
We will show that 
they do not interact
between each other
and keep the same number and order as time goes
on.

We define the open regions 
\begin{displaymath}
  \Omega_{0,j}=
  \left\{(t,x)\in\left[0,+\infty\right)\times \mathbb{R}
    \mid x<\bar x_{1}\land y_{j}(t)<x<y_{j+1}(t)\right\},
\end{displaymath}
and the flux
\begin{displaymath}
  F^{\varepsilon}\left(t,x,\sigma\right)=f^{0,j}\left(\sigma\right),\text{
  for } \left(t,x\right)\in\Omega_{0,j}.
\end{displaymath}
The wave front tracking approximation $s^{\varepsilon}$ so constructed
is an exact weak entropic solution to
\begin{displaymath}
  \begin{cases}
    \partial_{t}s^{\varepsilon}+\partial_{x}\left[F^{\varepsilon}\left(t,x,s^{\varepsilon}\right)\right]=0,\\
    \partial_{t}\left(c^{\varepsilon}s^{\varepsilon}\right)
    +\partial_{x}\left[c^{\varepsilon}F^{\varepsilon}\left(t,x,s^{\varepsilon}\right)\right]_{x}=0,\\
  \end{cases}
\end{displaymath}
in the region $x<\bar x_{1}$.

Since the $c$ family is linearly degenerate, $c$ waves will not interact with each other. 
Indeed, given two consecutive $c$ waves located respectively in $y_{j}(t)$
and $y_{j+1}(t)$, the first conservation law in the previous system
implies
\begin{displaymath}
  \begin{split}
    \frac{d}{dt}\int_{y_{j}(t)}^{y_{j+1}(t)}s^{\varepsilon}\left(t,x\right)\;dx
    &= \dot
    y_{j+1}\left(t\right)s^{\varepsilon}\left(t,y_{j+1}\left(t\right)-\right)
    -\dot
    y_{j}\left(t\right)s^{\varepsilon}\left(t,y_{j}\left(t\right)+\right)\\
    &\quad-f^{0,j}\left(s^{\varepsilon}\left(t,y_{j+1}\left(t\right)-\right)\right)
      +f^{0,j}\left(s^{\varepsilon}\left(t,y_{j}\left(t\right)+\right)\right)=0
      \end{split}
\end{displaymath}
since $\dot y = \lambda_{c}=\frac{f}{\sigma}$. Hence $c$ waves cannot
interact with each other (if $s^{\varepsilon} =0$ between two $c$ waves, then
they both must travel with zero speed and therefore
even in this case they cannot interact). 

Since any interaction with the $k$ wave located at $\bar x_{1}$ 
cannot give rise to waves entering the region $x<\bar x_{1}$,
following~\cite{WSCauchy}, the wave front tracking algorithm can be
carried out for all times in that region. Observe that for a fixed
$\varepsilon$ the total variation of the singular variable $P$
introduced in~\eqref{eq:sig_trans} is bounded. Since the grid
$\mathcal{S}_{0,j}$ contains all the possible maximum points of
$g\left(\cdot ,c_{j},k_{0}\right)$, in the regions $\Omega_{0,j}$,
$P\left(\sigma,c_{j},k_{0}\right)=
\int_{0}^{\sigma}\left|\partial_{\xi}\frac{f^{0,j}\left(\xi\right)}{\xi}\right|d\xi$
for any $\sigma\in\mathcal{S}_{0,j}$. The variable $P$ is
well behaved in the interplay between the two resonant waves $s$ and
$c$. Unfortunately, this behavior is disrupted by the third family of
waves, the $k$ waves,
except in the case where very strong hypothesis are assumed on the
flux as in~\cite{CocliteRisebro}. In fact, in \cite{CocliteRisebro} it
is assumed that the point of
maximum of $g$ does not change with $k$ (actually in~\cite{CocliteRisebro}
our $k$ waves correspond to discontinuities in time because of  
Lagrangian coordinates). 
Since such assumptions are not realistic for our model,
we are not able to prove a bound on the total
variation of $P$ uniformly in $\varepsilon$. 
Instead, we resolve this difficulty by applying a compensated compactness argument. 

Up to now, all the values of $\left(s^{\varepsilon},c^{\varepsilon},
  k^{\varepsilon}\right)$ are determined for $x < \bar x_{1}$. 
Since $k^{\varepsilon}$ is constant in time, its value at the
right of $\bar x_{1}$ is known. The jump conditions $\Delta f=\Delta
c=0$ determine all the values of
$\left(s^{\varepsilon},c^{\varepsilon},
  k^{\varepsilon}\right)$ to the right of
$\bar x_{1}$. These values could introduce new values for the function
$g$ that must be added to the grid, i.e.,
\begin{displaymath}
  \mathcal{G}_{1}=\mathcal{G}_{0}\cup
  \left\{g\left(s^{\varepsilon}\left(t,\bar x_{1}+\right),
      c^{\varepsilon}\left(t,\bar x_{1}+\right),
      k^{\varepsilon}\left(\bar x_{1}+\right)\right)\mid\; t\ge 0\right\}.
\end{displaymath}
This gives the new allowed values for $s$ for  $\bar x_{1}<x<\bar x_{2}$:
\begin{displaymath}
  \mathcal{S}_{1,j}=\left\{\sigma\mid g\left(\sigma,c_{j},k_{1}\right)\in
    \mathcal{G}_{1}\right\}.
\end{displaymath}
Using these values we now build the corresponding approximations $f^{1,j}\left(\sigma\right)$
of the flux as the
linear interpolation of $f\left(\sigma,c_{j},k_{1}\right)$ according to the
points in $\mathcal{S}_{1,j}$. Then we solve, as before, all the
Riemann problems at $t=0$, $\bar x_{1}<x<\bar x_{2}$ and $t\ge 0$,
$x=\bar x_{1}$. As before $c$ waves cannot interact with each other
so,
by induction, we can carry out the wave front tracking algorithm on
the semi plane $t\ge 0$.

We define the open regions ($\bar
x_{0}=y_{0}(t)=-\infty$, $\bar x_{N+1}=y_{M+1}(t)=+\infty$)
\begin{displaymath}
  \Omega_{i,j}=
  \left\{(t,x)\in\left[0,+\infty\right)\times \mathbb{R}
    \mid \bar x_{i}<x<\bar x_{i+1}\land y_{j}(t)<x<y_{j+1}(t)\right\}.
\end{displaymath}
The wave front tracking approximations so obtained are
weak entropic solutions to
\begin{equation}
  \label{eq:approxeq}
  \begin{cases}
    \partial_{t}s^{\varepsilon}+\partial_{x}\left[F^{\varepsilon}\left(t,x,s^{\varepsilon}\right)\right]=0,\\
    \partial_{t}\left(c^{\varepsilon}s^{\varepsilon}\right)
    +\partial_{x}\left[c^{\varepsilon}F^{\varepsilon}\left(t,x,s^{\varepsilon}\right)\right]=0,\\
    \partial_{t}k^{\varepsilon}=0,\\
    \left(s^{\varepsilon},c^{\varepsilon},k^{\varepsilon}\right)(0,x)=
    \left(\bar s^{\varepsilon},\bar c^{\varepsilon},\bar
      k^{\varepsilon}\right)(x),
  \end{cases}
\end{equation}
where the flux $F^{\varepsilon}$ is defined by
\begin{displaymath}
  F^{\varepsilon}\left(t,x,\sigma\right)=f^{i,j}\left(\sigma\right),\text{
  for } \left(t,x\right)\in\Omega_{i,j}.
\end{displaymath}
For all $\left(t,x\right)\in
  \left[0,+\infty\right[\times\mathbb{R}\text{ and
  }\sigma\in\left[0,1\right]$, the  flux satisfies the estimates
  \begin{equation}
    \label{eq:unifEstimates}
  \begin{cases}
    \left|F^{\varepsilon}\left(t,x,\sigma\right)-f\left(\sigma,c^{\varepsilon}\left(t,x\right),k^{\varepsilon}\left(x\right)\right)\right|\le
    \varepsilon,\\
    \left|\partial_{\sigma}F^{\varepsilon}\left(t,x,\sigma\right)-
      \partial_{\sigma}f\left(\sigma,c^{\varepsilon}\left(t,x\right),k^{\varepsilon}\left(x\right)\right)\right|\le
    \frac{\varepsilon}{N+M}.
  \end{cases}
\end{equation}
We remark that, in any region
$\Omega_{i,j}$, $s^{\varepsilon}$ is an entropic
solution
to the scalar conservation law
\begin{displaymath}
  \partial_{t}s^{\varepsilon}+\partial_{x}\left[f^{i,j}\left(s^{\varepsilon}\right)\right]=0.
\end{displaymath}

\section{Entropy estimates}
\label{sec:entropy}
Given a smooth (not necessarily convex) entropy function
$\eta\left(\sigma\right)$ with $\eta(0)=0$, we 
define the corresponding entropy flux $q^{\varepsilon}$
(relative to the approximate flux $F^{\varepsilon}$) as
\begin{equation}
  \label{eq:entflux}
      q^{\varepsilon}\left(t,x,\sigma\right) = \int_{0}^{\sigma}\eta'\left(\varsigma\right)
    \partial_{\varsigma}F^{\varepsilon}\left(t,x,\varsigma\right)d\varsigma.
\end{equation}
\begin{theorem}
  \label{thm:upperbound}
  For a fixed \textbf{convex} entropy $\eta$, 
  the positive part of the measure
  \begin{equation}
    \label{eq:entropy_ineq}
    \mu_{\varepsilon}=
    \partial_{t}\left[\eta\left(s^{\varepsilon}\right)\right]+\partial_{x}\left[q^{\varepsilon}\left(t,x,s^{\varepsilon}\right)\right]
  \end{equation}
  is uniformly (with
  respect to the approximation parameter $\varepsilon$)
  locally bounded.
  More precisely, for any compact set $K\subset
  \left]0,+\infty\right[\times\mathbb{R}$ there exists a constant
  $C_{K}$ such that
  \begin{displaymath}
    \mu_{\varepsilon}^{+}\left(K\right)\le C_{K}.
  \end{displaymath}
Here the constant $C_K$ may depend on $\eta$, $f$ and on the  total
  variation of the initial data $\bar c$ and $\bar k$, but it
  does not depend on the approximation
  parameter $\varepsilon$. 
\end{theorem}
\begin{proof}
  Fixing a non negative test function $\phi\in
  C^{\infty}_{c}\left(\left]0,+\infty\right[\times \mathbb{R}\right)$,
  we compute
  \begin{equation}
    \label{eq:sum_waves}
    \begin{split}
      \langle
      \partial_{t}\left[\eta\left(s^{\varepsilon}\right)\right]+\partial_{x}\left[q^{\varepsilon}
      \left(t,x,s^{\varepsilon}\right)\right],\phi\rangle
      &= -\int
      \eta\left(s^{\varepsilon}\right)\partial_{t}\phi+
      q^{\varepsilon}\left(t,x,s^{\varepsilon}\right)\partial_{x}\phi\;
      dtdx\\
      &=\int_{0}^{+\infty}\sum_{\ell=1}^{\mathcal{N}(t)}
      \left[\Delta q^{\varepsilon}_{\ell}\left(t\right)-\Delta\eta_{\ell}\left(t\right)\dot
        \xi_{\ell}\left(t\right) \right]
      \phi\left(t,\xi_{\ell}\left(t\right)\right)dt.
    \end{split}
  \end{equation}
 Here $\xi_{1},\ldots,\xi_{\mathcal{N}(t)}$ are the locations of the
  discontinuities in
  $\left(s^{\varepsilon},c^{\varepsilon},k^{\varepsilon}\right)$,
  and the notation $\Delta$ denotes the jumps:
  \begin{displaymath}
    \begin{cases}
      \Delta\eta_{\ell}\left(t\right)=
      \eta\left(s^{\varepsilon}\left(t,\xi_{\ell}\left(t\right)+\right)\right)
      -\eta\left(s^{\varepsilon}\left(t,\xi_{\ell}\left(t\right)-\right)\right),\\
      \Delta q^{\varepsilon}_{\ell}\left(t\right)=
      q^{\varepsilon}\left(t,\xi_{\ell}(t)+,
        s^{\varepsilon}\left(t,\xi_{\ell}\left(t\right)+\right)\right)-
      q^{\varepsilon}\left(t,\xi_{\ell}(t)-,
        s^{\varepsilon}\left(t,\xi_{\ell}\left(t\right)-\right)\right).
    \end{cases}
  \end{displaymath}
  We study separately the three different kinds of waves
  and denote with the superscripts ``$-$'' and ``$+$'' the
  values computed respectively to the left and the right of the
  discontinuities. We omit the superscript for the  
  values that do not change across the discontinuities.

\medskip
\noindent\textbf{$s$ waves:} Both $c$ and $k$ are
constant while the jump in $s$ satisfies Rankine-Hugoniot and is
entropic according to the approximate flux. If
$\left(t,\xi_{\ell}\left(t\right)\right)\in\Omega_{i,j}$, then
\begin{displaymath}
  \dot
  \xi_{\ell}\left(s^{+}-s^{-}\right)=f^{i,j}\left(s^{+}\right)-f^{i,j}\left(s^{-}\right)
  =f^{+}-f^{-}.
\end{displaymath}
Hence, applying the definition of $q^{\varepsilon}$ and
integrating by parts, we  compute
\begin{displaymath}
  \begin{split}
    \Delta
    q^{\varepsilon}_{\ell}-\Delta\eta_{\ell}\dot
    \xi_{\ell}
    &=
    \int_{s^{-}}^{s^{+}}\eta'\left(\varsigma\right)
    \left[\partial_{\varsigma}f^{i,j}\left(\varsigma\right)-\dot
      \xi_{\ell}\right]d\varsigma \\
    &=\left[\eta'\left(\varsigma\right)\left(f^{i,j}\left(\varsigma\right)-f^{-}
      -
      \dot\xi_{\ell}\left(\varsigma-s^{-}\right)\right)\right]_{s^{-}}^{s^{+}}\\
  &\qquad\qquad -\int_{s^{-}}^{s^{+}}\eta''\left(\varsigma\right)
  \left[f^{i,j}\left(\varsigma\right)-f^{-}-\dot\xi_{\ell}\left(\varsigma-s^{-}\right)\right]
  d\varsigma\\
  &\le 0.
  \end{split}
\end{displaymath}
Since $\eta''\ge 0$ and the $s$ wave in $\xi_{\ell}$ is an entropic
wave for the flux $f^{i,j}$, therefore for all
$\varsigma\in\left[\min\left\{s^{-},s^{+}\right\},
\max\left\{s^{-},s^{+}\right\}\right]$ one has
\begin{displaymath}
\sign\left(s^{+}-s^{-}\right)
\left[f^{i,j}\left(\varsigma\right)-f^{-}-
  \frac{f^{+}-f^{-}}{s^{+}-s^{-}}\left(\varsigma-s^{-}\right)\right]
\ge 0.
\end{displaymath}

\medskip
\noindent\textbf{$c$ waves:}  Both $k$ and
$g=\frac{f}{s}$ are constants and the speed $\dot \xi_{\ell}$ of the
wave equals
$g\left(s^{-},c^{-},k\right)=g\left(s^{+},c^{+},k\right)$, where
 $\xi_{\ell}$ is the boundary between the regions $\Omega_{i,j}$
and $\Omega_{i,j+1}$. 
Denoting by $C$ a generic constant that depends
only on $\eta$ and $f$,  the
uniform estimates~\eqref{eq:unifEstimates}  lead to
\begin{displaymath}
  \begin{split}
    \Delta
    &q^{\varepsilon}_{\ell}-\Delta\eta_{\ell}\dot
    \xi_{\ell}
    =
    \int_{0}^{s^{+}}\eta'\left(\varsigma\right)
    \partial_{\varsigma}f^{i,j+1}\left(\varsigma\right)d\varsigma
    -\int_{0}^{s^{-}}\eta'\left(\varsigma\right)
    \partial_{\varsigma}f^{i,j}\left(\varsigma\right)d\varsigma
      -\dot \xi_{\ell}\left(\eta\left(s^{+}\right)-\eta\left(s^{-}\right)\right)\\
    &\le C\frac{\varepsilon}{N+M}+
    \int_{0}^{s^{+}}\eta'\left(\varsigma\right)
    \partial_{\varsigma}f\left(\varsigma,c^{+},k\right)d\varsigma
    -\int_{0}^{s^{-}}\eta'\left(\varsigma\right)
    \partial_{\varsigma}f\left(\varsigma,c^{-},k\right)d\varsigma\\
    &\qquad
    -\int_{0}^{s^{+}}\eta'\left(\varsigma\right)\dot\xi_{\ell}d\varsigma
    +\int_{0}^{s^{-}}\eta'\left(\varsigma\right)\dot\xi_{\ell}d\varsigma\\
    &\le C\frac{\varepsilon}{N+M}+
    \int_{0}^{s^{+}}\eta'\left(\varsigma\right)
    \left[\partial_{\varsigma}f\left(\varsigma,c^{+},k\right)-\dot \xi_{\ell}\right]d\varsigma
    -\int_{0}^{s^{-}}\eta'\left(\varsigma\right)
    \left[\partial_{\varsigma}f\left(\varsigma,c^{-},k\right)
    -\dot \xi_{\ell}\right]d\varsigma\\
  & \le C\frac{\varepsilon}{N+M}-
    \int_{0}^{s^{+}}\eta''\left(\varsigma\right)
    \left[f\left(\varsigma,c^{+},k\right)-\dot \xi_{\ell}\varsigma\right]d\varsigma
    +\int_{0}^{s^{-}}\eta''\left(\varsigma\right)
    \left[f\left(\varsigma,c^{-},k\right)
    -\dot \xi_{\ell}\varsigma\right]d\varsigma.\\
\end{split}
\end{displaymath}
Here we have integrated by parts and used the relations
\[f\left(0,c^{\pm},k\right)=0, \qquad f\left(s^{\pm},c^{\pm},k\right)
=s^{\pm}g\left(s^{\pm},c^{\pm},k\right)=s^{\pm}\dot \xi_{\ell}.\]
Suppose $s^{-}\le s^{+}$, the other case being symmetric. Because of
the entropy condition on
$c$ waves~\eqref{eq:entropy2} there exists
$s^{*}\in\left[s^{-},s^{+}\right]$ such that 
\begin{displaymath}
  \begin{cases}
    g\left(\varsigma,c^{-},k\right)\ge \dot\xi_{\ell}&\text{ for all
    }\varsigma
    \in\left[s^{-},s^{*}\right],\\
    g\left(\varsigma,c^{+},k\right)\ge \dot\xi_{\ell}&\text{ for all
    }\varsigma
    \in\left[s^{*},s^{+}\right].
  \end{cases}
\end{displaymath}
The estimates~\eqref{eq:unifEstimates}  further lead to
\begin{displaymath}
  \begin{split}
    \Delta
    q^{\varepsilon}_{\ell}-\Delta\eta_{\ell}\dot
    \xi_{\ell}
  & \le-
    \int_{s^{*}}^{s^{+}}\eta''\left(\varsigma\right)
    \left[f\left(\varsigma,c^{+},k\right)-\dot \xi_{\ell}\varsigma\right]d\varsigma
    +\int_{s^{*}}^{s^{-}}\eta''\left(\varsigma\right)
    \left[f\left(\varsigma,c^{-},k\right)
      -\dot \xi_{\ell}\varsigma\right]d\varsigma\\
    &\qquad
    +C\left(\frac{\varepsilon}{N+M}+\left|c^{+}-c^{-}\right|\right)\\
  & =-
    \int_{s^{*}}^{s^{+}}\eta''\left(\varsigma\right)
    \varsigma
    \left[g\left(\varsigma,c^{+},k\right)-\dot \xi_{\ell}\right]d\varsigma
    +\int_{s^{*}}^{s^{-}}\eta''\left(\varsigma\right)
    \varsigma\left[g\left(\varsigma,c^{-},k\right)
      -\dot \xi_{\ell}\right]d\varsigma\\
    &\qquad
    +C\left(\frac{\varepsilon}{N+M}+\left|c^{+}-c^{-}\right|\right)\\
    &\le
    C\left(\frac{\varepsilon}{N+M}+\left|\Delta c_{\ell}\right|\right).
  \end{split}
\end{displaymath}

\medskip
\noindent\textbf{$k$ waves:} For a $k$ wave, both $c$ and $f$ are constant and
$\dot \xi_{\ell}=0$, where
$\xi_{\ell}$ is the boundary between two regions $\Omega_{i,j}$
and $\Omega_{i+1,j}$.  We have
\begin{displaymath}
  \begin{split}
    \Delta
    &q^{\varepsilon}_{\ell}-\Delta\eta_{\ell}\dot
    \xi_{\ell}
    =
    \int_{0}^{s^{+}}\eta'\left(\varsigma\right)
    \partial_{\varsigma}f^{i+1,j}\left(\varsigma\right)d\varsigma
    -\int_{0}^{s^{-}}\eta'\left(\varsigma\right)
    \partial_{\varsigma}f^{i,j}\left(\varsigma\right)d\varsigma
      \\
    &\le C\frac{\varepsilon}{N+M}+
    \int_{0}^{s^{+}}\eta'\left(\varsigma\right)
    \partial_{\varsigma}f\left(\varsigma,c,k^{+}\right)d\varsigma
    -\int_{0}^{s^{-}}\eta'\left(\varsigma\right)
    \partial_{\varsigma}f\left(\varsigma,c,k^{-}\right)d\varsigma\\
    &\le C\left(\frac{\varepsilon}{N+M}+\left|k^{+}-k^{-}\right|\right)+
    \int_{s^{-}}^{s^{+}}\eta'\left(\varsigma\right)
    \partial_{\varsigma}f\left(\varsigma,c,k^{-}\right)d\varsigma
\\
    &\le C\left(\frac{\varepsilon}{N+M}+\left|\Delta k_{\ell}\right|\right)+
    \left\|\eta'\right\|_{\infty}\sign\left(s^{+}-s^{-}\right)\int_{s^{-}}^{s^{+}}
    \partial_{\varsigma}f\left(\varsigma,c,k^{-}\right)d\varsigma
\\
    &= C\left(\frac{\varepsilon}{N+M}+\left|\Delta k_{\ell}\right|+
    \left| f\left(s^{+},c,k^{-}\right)-f\left(s^{-},c,k^{-}\right)\right|\right)
\\
    &= C\left(\frac{\varepsilon}{N+M}+\left|\Delta k_{\ell}\right|+
    \left| f\left(s^{+},c,k^{-}\right)-f\left(s^{+},c,k^{+}\right)\right|\right)
\\
    &= C\left(\frac{\varepsilon}{N+M}+\left|\Delta k_{\ell}\right|\right)
\end{split}
\end{displaymath}
where we used the fact that
$\partial_{\varsigma}f\left(\varsigma,c,k^{-}\right)\ge 0$ and that
$f\left(s^{-},c,k^{-}\right) = f\left(s^{+},c,k^{+}\right)$.

\bigskip

Finally, if the compact support of $\phi$ is contained in
$\left]0,T\right[\times\mathbb{R}$, equality~\eqref{eq:sum_waves} and
the previous analysis on the three
types of waves lead to
\begin{displaymath}
      \begin{split}
      \langle
      \partial_{t}\left[\eta\left(s^{\varepsilon}\right)\right]+\partial_{x}\left[q^{\varepsilon}
      \left(t,x,s^{\varepsilon}\right)\right],\phi\rangle
      &\le CT\left(\frac{N\varepsilon + M\varepsilon}{N+M} + \tv\left\{\bar c\right\}
        +\tv\left\{\bar k\right\}\right)
      \left\|\phi\right\|_{\infty}\\
      &\le CT\left(1 + \tv\left\{\bar c\right\}
        +\tv\left\{\bar k\right\}\right)
      \left\|\phi\right\|_{\infty}
    \end{split}
  \end{displaymath}
  for any $\varepsilon\in\left]0,1\right[$, proving the theorem.
\end{proof}

\begin{theorem}
  \label{thm:compactness}
  For any smooth entropy $\eta$ (even non convex) and decreasing
  sequence $\varepsilon_{j}\to 0$ there exists a
  compact set $\mathcal{K}\subset H^{-1}_{loc}\left(\Omega\right)$,
  independent of ${j}$,
  such that
  \begin{displaymath}
        \mu_{\varepsilon_{j}}=
        \partial_{t}\left[\eta\left(s^{\varepsilon_{j}}\right)\right]+\partial_{x}\left[q^{\varepsilon_{j}}
        \left(t,x,s^{\varepsilon_{j}}\right)\right]\in\mathcal{K}.
  \end{displaymath}
\end{theorem}
\begin{proof}
  We apply standard arguments in compensated compactness
  theory~\cite{diperna}. 
  Integrating the measure $\mu_{\varepsilon}$ over a rectangle (with $t_{1}>0$)
  $R=\left[t_{1},t_{2}\right]\times \left[-L,L\right]$ we obtain
  \begin{displaymath}
    \begin{split}
      \mu_{\varepsilon}\left(R\right)&=\int_{t_{1}}^{t_{2}}
      q^{\varepsilon}\left(t,L+,s^{\varepsilon}\left(t,L+\right)\right)
      -q^{\varepsilon}\left(t,-L-,s^{\varepsilon}\left(t,-L-\right)\right)dt\\
&\qquad +\int_{-L}^{L}
      \eta\left(s^{\varepsilon}\left(t_{2}+,x\right)\right)-
      \eta\left(s^{\varepsilon}\left(t_{1}-,x\right)\right)dx.
    \end{split}
  \end{displaymath}
  Since $s^{\varepsilon}$ is uniformly bounded, there exists a
  constant $\bar C_{R}$ such that
  $\left|\mu_{\varepsilon}\left(R\right)\right|\le \bar C_{R}$ for any
  $\varepsilon\in\left]0,1\right[$. If   $\eta$ is
  convex,  we can apply 
  Theorem~\ref{thm:upperbound}  to estimate the total
  variation of $\mu_{\varepsilon}$ uniformly with respect to $\varepsilon$:
  \begin{displaymath}
    \left|\mu_{\varepsilon}\right|\left(R\right) =
    \mu_{\varepsilon}^{+}\left(R\right)+
    \mu_{\varepsilon}^{-}\left(R\right)=
    2\mu_{\varepsilon}^{+}\left(R\right) -
    \mu_{\varepsilon}\left(R\right)
    \le 2C_{R} + \bar C_{R}.
  \end{displaymath}
  If $\eta$ is not convex, then we take a strictly convex entropy $\eta^{*}$
  (for instance $\eta^{*}\left(\sigma\right)=\sigma^{2}$) and define
  $\tilde \eta = \eta + H \eta^{*}$. The entropy $\tilde \eta$
  is convex for a sufficiently big constant $H$. 
  We denote by $\mu_{\varepsilon}$,
  $\mu_{\varepsilon}^{*}$ and $\tilde \mu_{\varepsilon}$ the
  measures corresponding to the entropies $\eta$, $\eta^{*}$ and $\tilde
  \eta$. 
  Since the definition of the entropy
  flux~\eqref{eq:entflux} is linear with respect to the associated
  entropy, 
  the measures satisfy $\tilde
  \mu_{\varepsilon}=\mu_{\varepsilon}+H\mu_{\varepsilon}^{*}$. Hence
  the inequality 
  \begin{displaymath}
    \left|\mu_{\varepsilon}\right|\left(R\right)\le\left|\tilde
      \mu_{\varepsilon}\right|\left(R\right) 
    + H \left|\mu^{*}_{\varepsilon}\right|\left(R\right)
  \end{displaymath}
  holds.
  This means that $\left|\mu_{\varepsilon}\right|\left(R\right)$ is
  bounded 
  uniformly with respect to $\varepsilon$ since both
  $\tilde \mu_{\varepsilon}$ and $\mu^{*}_{\varepsilon}$ are
  associated with convex entropies.
  Since the measure $\mu_{\varepsilon}=
    \partial_{t}\left[\eta\left(s^{\varepsilon}\right)\right]+\partial_{x}\left[q^{\varepsilon}\left(t,x,s^{\varepsilon}\right)\right]
$ restricted to $R$ lies both in a bounded set of the space of
measures $\mathcal{M}\left(R\right)$ and in a bounded set of
$W^{-1,\infty}\left(R\right)$, \cite[Lemma~17.2.2]{Dafermos} allows us
to conclude the proof of the theorem. 
\end{proof}

\section{Strong Convergence}
\label{sec:conv}
The following result is a step 
towards the proof of Theorem~\ref{thm:main}.
\begin{theorem}
  \label{thm:strongconvergence}
  There exists a sequence $\varepsilon_{j}\to 0$ such that
  $\left(s^{\varepsilon_{j}},c^{\varepsilon_{j}},k^{\varepsilon_{j}}\right)
  \to \left(\tilde s,\tilde c,\tilde k\right)$ in $L^{1}_{loc}\left(\Omega\right)$.
\end{theorem}
\begin{proof}
  We suitably modify the proof of~\cite[Theorem~4.2]{BGS},  omitting
  some computations already written there.
  The proof takes several steps.
  
  \paragraph{1.}
  Observe that by construction we have
  \begin{displaymath}
  \tv\left\{c^{\varepsilon}\left(t,\cdot\right)\right\}
  = \tv\left\{\bar c^{\varepsilon}\right\}\le\tv\left\{\bar
    c\right\}
\end{displaymath}
  and the wave speeds are uniformly bounded.  
  Hence Helly's theorem implies that there exist a sequence
  $c^{\varepsilon_{j}}\to \tilde c$ in
  $L^{1}_{loc}\left(\Omega\right)$. Since $k^{\varepsilon}$ is
  constant in time,  we have $k^{\varepsilon_{j}}\to \tilde k = \bar k$ in
  $L^{1}_{loc}\left(\Omega\right)$ as well. In the following we
  always take subsequences of this sequence and we will drop the index
  $j$ to simplify notations. 
  We define the limit flux
  \begin{displaymath}
    F\left(t,x,\sigma\right)=f\left(\sigma,\tilde c(t,x),
      \tilde k(x)\right),\quad
    \text{ for all }\left(t,x\right)\in\Omega,\text{ and }\sigma\in\left[0,1\right]
  \end{displaymath}
  and for any entropy $\eta$ we define the limit entropy flux
  \begin{displaymath}
    q\left(t,x,\sigma\right)=\int_{0}^{\sigma}\eta'\left(\varsigma\right)
    \partial_{\varsigma}F\left(t,x,\varsigma\right)d\varsigma.
  \end{displaymath}
  The estimate (uniform  in $\sigma\in\left[0,1\right]$) 
  \begin{displaymath}
    \begin{split}
      &\left|q\left(t,x,\sigma\right)-q^{\varepsilon}\left(t,x,\sigma\right)\right|
      \le \int_{0}^{1}\left|\eta'\left(\varsigma\right)\right|
      \bigg(\left|\partial_{\varsigma}f\left(\varsigma,\tilde c(t,x),
          \tilde k(x)\right)-\partial_{\varsigma}f\left(\varsigma,c^{\varepsilon}(t,x),
          k^{\varepsilon}(x)\right)\right|\\
      &\qquad\qquad\qquad\qquad\qquad\qquad\qquad\qquad\qquad
      +\left|\partial_{\varsigma}f\left(\varsigma,c^{\varepsilon}(t,x),
          k^{\varepsilon}(x)\right)-\partial_{\varsigma}F^{\varepsilon}
        \left(t,x,\varsigma\right)\right|\bigg)d\varsigma\\
      &\quad \le
      C\left(\left|\tilde
          c(t,x)-c^{\varepsilon}(t,x)\right|+\left|\tilde k(x)-k^{\varepsilon}
          (x)\right|+\varepsilon\right)\to 0\qquad
      \text{ in }L^{1}_{loc}\left(\Omega\right)
    \end{split}
  \end{displaymath}
  implies that 
  \[\partial_{x}\left[    q\left(t,x,s^{\varepsilon}\right)-
    q^{\varepsilon}\left(t,x,s^{\varepsilon}\right)
  \right]\to 0, \quad \mbox{ in} ~H^{-1}_{loc}\left(\Omega\right).\]
  Together with Theorem~\ref{thm:compactness}, it implies that the sequence
  \begin{displaymath}
    \partial_{t}\left[\eta\left(s^{\varepsilon}\right)\right]+
    \partial_{x}\left[q\left(t,x,s^{\varepsilon}\right)\right]=
    \partial_{t}\left[\eta\left(s^{\varepsilon}\right)\right]+
    \partial_{x}\left[q^{\varepsilon}\left(t,x,s^{\varepsilon}\right)\right]
    +
\partial_{x}\left[    q\left(t,x,s^{\varepsilon}\right)-
    q^{\varepsilon}\left(t,x,s^{\varepsilon}\right)
  \right]
\end{displaymath}
belongs to a compact set in $H^{-1}_{loc}\left(\Omega\right)$.
\paragraph{2.}
  For any $(t,x)\in\Omega$ and $v,w\in[0,1]$ we define
  \begin{equation}
    \label{eq:jensen_def}
    I(t,x,v,w)~\doteq~(v-w)\int_{w}^{v}\bigl[\partial_{\sigma}F
    (t,x,\sigma)\bigr]^{2}\,d\sigma
    -\bigl[F(t,x,v)-F(t,x,w)\bigr]^{2}.
  \end{equation}
  The following properties hold.
  \begin{enumerate}[(i)]
  \item
    $(v,w)\mapsto I(t,x,v,w)$ is continuous with
    $I(t,x,v,v)=0$ for any $v\in\bigl[0,1\bigr]$.
  \item
    $I(t,x,v,w)>0$ for any
    $v,w\in\bigl[0,1\bigr]$ with $v\not=w$.
  \end{enumerate}
  Indeed, (i) is trivial, while (ii) follows from Jensen's
  inequality and the fact that
  $\sigma\mapsto f\left(\sigma,\gamma,\kappa\right)$ and hence
  $\sigma\mapsto F\left(t,x,\sigma\right)$ have a
    unique inflection point.
  Indeed suppose $w<v$, we observe that $\sigma\mapsto \partial_{\sigma}F(t,x,\sigma)$ is not
  constant  over the interval $\omega\in[w,v]$, and we compute  
  \begin{align*}
    I(t,x,v,w)~&  =~(v-w)\int_w^v\bigl[\partial_{\sigma}F(t,x,\sigma)\bigr]^{2}\,d\sigma
    -(v-w)^{2}\left[\frac{1}{v-w}\int_w^v
    \partial_{\sigma}F(t,x,\sigma)\,d\sigma\right]^{2}\\
               &>~(v-w)\int_w^v\bigl[\partial_{\sigma}F
                 (t,x,\sigma)\bigr]^{2}\,d\sigma
                 -(v-w)^{2}\frac{1}{v-w}\int_w^v
                 \bigl[\partial_{\sigma}F(t,x,\sigma)\bigr]^{2}\,d\sigma\\
                           &=~0.
  \end{align*}
  
\paragraph{3.}
Fixing
  $(\tau,y)\in\Omega$ and we consider the following entropies and
  corresponding limit fluxes
  \begin{align*}
    \eta(\sigma)&=\sigma,
    &q(t,x,\sigma)&=
                                  F(t,x,\sigma),\\
    \eta_{\left(\tau,y\right)}(\sigma)&=F(\tau,y,\sigma),&
    q_{\left(\tau,y\right)}(t,x, \sigma)
                                            &=\int_{0}^{
                                              \sigma}\partial_{\varsigma}F(\tau,y,
                                              \varsigma)
                                              \partial_{\varsigma}F
                                              (t,x,\varsigma)\,d\varsigma.
  \end{align*}
  The same computations as the ones used to obtain~\cite[(4.16)]{BGS}
  prove that there exists a constant $C_{2}\ge 0$ such that
  \begin{multline}
    \label{eq:q2_and_I}
    (v-w)\bigl[q_{\left(\tau,y\right)}(t,x,v)-q_{\left(\tau,y\right)}(t,x,w)\bigr]\\\ge~
    I(t,x,v,w)+\bigl[F(t,x,v)-F(t,x,w)\bigr]^{2}
                      -C_{2}\sup_{\sigma\in[0,1]}
                      \bigl|F(\tau,y,\sigma)-F(t,x,\sigma)\bigr|.
  \end{multline}

  \paragraph{4.}
  By possibly taking subsequences, we can achieve the following
 weak$^{*}$ 
 convergences in
  $L^{\infty}(\Omega)$:
  \begin{equation}
    \label{eq:weak_limits}
    \begin{cases}
    \displaystyle    s^{\varepsilon}(t,x)
    ~\overset{*}{\rightharpoonup}~ 
    \tilde s(t,x),\\[1mm]
    \displaystyle   F
    \bigl(t,x,s^{\varepsilon}(t,x)\bigr)~\overset{*}\rightharpoonup~
                                             \tilde F(t,x),\\[1mm]
\displaystyle    I\bigl(t,x,s^{\varepsilon}(t,x),\tilde
    s(t,x)\bigr)~\overset{*}\rightharpoonup~
                              \tilde I(t,x).
\end{cases}
\end{equation}
  Taking further subsequences (which this time may depend on
  $\left(\tau,y\right)$) we can achieve these further weak$^{*}$ convergences in
  $L^{\infty}(\Omega)$ 
  \begin{align}
    \label{eq:weak_limits_bis}
    F\bigl(\tau,y,s^{\varepsilon}(t,x)\bigr)&~\overset{*}\rightharpoonup~
                                              \tilde F_{\left(\tau,y\right)}(t,x),
    &q_{\left(\tau,y\right)}
      \bigl(t,x,s^{\varepsilon}(t,x)\bigr)&~\overset{*}\rightharpoonup~
                                                          \tilde
                 q_{\left(\tau,y\right)}(t,x).
  \end{align}
  Notice that the weak limits $\tilde s$, $\tilde f$, $\tilde I$ in~\eqref{eq:weak_limits}
  do not
  depend on the values $\left(\tau,y\right)$.
  Step~\textbf{1} implies
  \begin{displaymath}
    \partial_{t}\left[s^{\varepsilon}(t,x)\right]+\partial_{x}
    \left[F\bigl(t,x,s^{\varepsilon}(t,x)\bigr)\right]
    ,\quad
    \partial_{t}\left[F\bigl(\tau,y,s^{\varepsilon}(t,x)\bigr)\right]+
    \partial_{x}\left[q_{\left(\tau,y\right)}\bigl(t,x,s^{\varepsilon}(t,x)\bigr)\right]~\in~\mathcal{K},
  \end{displaymath}
  where $\mathcal{K}$ is a compact set (independent of the subsequence
  index) in
  $H_{loc}^{-1}(\Omega)$. By an application of the
  \emph{div--curl lemma}, see for example  Theorem~16.2.1 in~\cite{Dafermos}, 
 one obtains
  \begin{equation}\label{eq:divcurl_result}\begin{array}{l}
    s^{\varepsilon}(t,x)q_{\left(\tau,y\right)}\bigl(t,x,s^{\varepsilon}(t,x)\bigr)
    -F\bigl(t,x,s^{\varepsilon}(t,x)\bigr)F\bigl(\tau,y,s^{\varepsilon}(t,x)\bigr)\\[3mm]
  \qquad\qquad   \overset{*}{\rightharpoonup}~
    \tilde s(t,x)\tilde q_{\left(\tau,y\right)}(t,x)-
    \tilde F(t,x)\tilde F_{\left(\tau,y\right)}(t,x).
  \end{array}\end{equation}
Following the proof of~\cite[Theorem~4.2]{BGS} we
set $v=s^{\varepsilon}(t,x)$ and $w=\tilde s(t,x)$
in \eqref{eq:q2_and_I} and take the weak$^{*}$ limit as
$\varepsilon\to 0$ to obtain
\[
  \begin{split}
    &\tilde I(t,x)-\left[\tilde s(t,x)\tilde q_{\left(\tau,y\right)}(t,x)-
      \tilde F(t,x)\tilde F_{\left(\tau,y\right)}(t,x)\right]+\tilde s(t,x)
    \tilde q_{\left(\tau,y\right)}(t,x)\\
    & \qquad -2 \tilde F(t,x)F\bigl(t,x,\tilde
    s(t,x)\bigr)+F\bigl(t,x,\tilde s(t,x)\bigr)^{2}~\le
    ~C_{3}\sup_{\sigma\in[0,1]}
    \bigl|F(\tau,y,\sigma)-F(t,x,\sigma)\bigr|.
  \end{split}
  \]
 This can be written as
 \begin{displaymath}
  \begin{split}
    \tilde I(t,x)+\bigl[\tilde F(t,x)-F\bigl(t,x,\tilde s(t,x)\bigr)\bigr]^{2}
    &\le C_{3}\sup_{\sigma\in[0,1]}
    \bigl|F(\tau,y,\sigma)-F(t,x,\sigma)\bigr|\\
    &\qquad+\bigl|\tilde
    F(t,x)\bigr|\bigl|\tilde F_{\left(\tau,y\right)}(t,x)-\tilde
    F(t,x)\bigr|,
  \end{split}
\end{displaymath}
  which holds for any fixed $(\tau,y)\in\Omega$ and
  a.e.~$(t,x)\in\Omega$.
  Taking the weak$^{*}$ limit in
\begin{eqnarray*}
    - \sup_{\sigma\in[0,1]}
    \bigl|F(\tau,y,\sigma)-F(t,x,\sigma)\bigr|
   & \le& F\bigl(\tau,y,s^{\varepsilon}(t,x)\bigr)-F\bigl(t,x,s^{\varepsilon}(t,x)\bigr)\\
   & \le& \sup_{\sigma\in[0,1]}
    \bigl|F(\tau,y,\sigma)-F(t,x,\sigma)\bigr|,
\end{eqnarray*}
  we obtain
  \begin{eqnarray*}
    - \sup_{\sigma\in[0,1]}
    \bigl|F(\tau,y,\sigma)-F(t,x,\sigma)\bigr|
  &  \le &\tilde F_{\left(\tau,y\right)}(t,x)-\tilde F(t,x)\\
   & \le&\sup_{\sigma\in[0,1]}
    \bigl|F(\tau,y,\sigma)-F(t,x,\sigma)\bigr|.
  \end{eqnarray*}
  Hence for any fixed $(\tau,y)\in\Omega$, we have for
  a.e.~$(t,x)\in\Omega$
  \begin{equation}
    \label{eq:last_ineq}
    \tilde I(t,x)+\bigl[\tilde F(t,x)-F\bigl(t,x,\tilde s(t,x)\bigr)\bigr]^{2}
   ~ \le~ C_{4}\sup_{\sigma\in[0,1]}
    \bigl|F(\tau,y,\sigma)-F(t,x,\sigma)\bigr|.
  \end{equation}

  \paragraph{4.} 
We call $E_{1}$  the set of Lebesgue points of the left hand side
  of~\eqref{eq:last_ineq}.  Moreover, for each $\sigma\in [0,1]$  let
  $E_{\sigma}$ 
  be the set of Lebesgue points
  of the map $(t,x) \mapsto F(t,x,\sigma)$. Defining
  \[
  E\doteq E_{1}\cap\left(\displaystyle{\bigcap_{q\in\mathbb{Q}\cap[0,1]}E_{q}}\right),
  \]
  we observe that its
  complement  $\Omega\setminus E$ has zero measure. Take any
  $(\tau,y)\in E$ and fix $\epsilon>0$. Let
  $\mathcal{F}_{\epsilon}\subset \mathbb{Q}\cap [0,1]$ be a finite set such that
  $\displaystyle{\inf_{q\in
      \mathcal{F}_{\epsilon}}}\bigl|q-\sigma\bigr|<\epsilon$
  for every $\sigma\in [0,1]$. Then we have
  \begin{eqnarray}
      \sup_{\sigma\in[0,1]}
      \bigl|F(\tau,y,\sigma)-F(t,x,\sigma)\bigr|
      &\le &\max_{q\in \mathcal{F}_{\epsilon}}
      \bigl|F(\tau,y,q)-F(t,x,q)\bigr| +
      2L\epsilon \nonumber\\
      &\le& \sum_{q\in \mathcal{F}_{\epsilon}}
      \bigl|F(\tau,y,q)-F(t,x,q)\bigr| +
      2L\epsilon,
        \label{eq:finite_ineq}
  \end{eqnarray}
  where $L$ is a uniform Lipchitz constant for $\varsigma\mapsto
  F\left(t,x,\varsigma\right)$.
  Let $B_{\delta}(\tau,y)$ be the disc in $\Omega$ centered
  in $(\tau,y)$ with radius $\delta>0$ whose area is $\pi \delta^2$.
 Integrating~\eqref{eq:last_ineq} and using~\eqref{eq:finite_ineq} we obtain
  \begin{displaymath}
    \begin{split}
      &\frac{1}{\pi \delta^2}\int_{B_{\delta}(\tau,y)} \Big(\tilde
      I(t,x)+\bigl[\tilde F(t,x)-F\bigl(t,x,\tilde
      s(t,x)\bigr)\bigr]^{2}\Big)\,dt\, dx\\
      &\qquad\qquad\qquad\le \frac{C_{4}}{\pi\delta^2} \sum_{q\in \mathcal{F}_{\epsilon}}
      \int_{B_{\delta}(\tau,y)} \bigl|F(\tau,y,q)-F(t,x,q)\bigr|\,dt\,
      dx+ 2C_{4}L\epsilon.
    \end{split}
  \end{displaymath}
  Since $(\tau,y)$ is a Lebesgue point for the map
  $(t,x)\mapsto F(t,x,q)$, for all $q\in \mathcal{F}_{\epsilon}$, letting
  $\delta\to 0$ we obtain
  \begin{displaymath}
    \tilde I(\tau,y)+\bigl[\tilde F(\tau,y)-F\bigl(\tau,y,\tilde s(\tau,y)\bigr)\bigr]^{2}
   ~ \le ~C_{4}L\epsilon\,.
  \end{displaymath}
 Since $\epsilon>0$ is arbitrary, this implies
  \begin{displaymath}
    \tilde I(\tau,y)+\bigl[\tilde F(\tau,y)-F\bigl(\tau,y,\tilde
        s(\tau,y)\bigr)\bigr]^{2}~\le~ 0\qquad \text{ for every~
    }(\tau,y)\in E\,.
  \end{displaymath}
 Hence $\tilde I(t,x)\le 0$ a.e.~in $\Omega$.
 Since, by Step~\textbf{2}, $I\bigl(t,x,s^{\varepsilon}(t,x),
 \tilde s(t,x)\bigr)\ge 0$,
  its weak$^{*}$ limit $\tilde I(t,x)$ must be greater or equal to
  zero almost everywhere.   Therefore we get
  \begin{displaymath}
    \tilde I(t,x)=0,\qquad\text{ and }\qquad \tilde
    F\left(t,x\right)=F\left(t,x,\tilde
      s\left(t,x\right)\right),\quad\text{ a.e. in }\Omega.
  \end{displaymath}
  Since  $I(t,x,s^{\varepsilon}(t,x),\tilde
  s(t,x)\bigr)\geq 0$ converges weakly$^{*}$ to
  zero, we conclude that  it  converges  
  strongly in $L^1_{loc}\left(\Omega\right)$. We can thus
  take a subsequence such that $I(t,x,s^{\varepsilon}(t,x),\tilde
    s(t,x)\bigr)\to 0$ a.e.~in $\Omega$. 
    Finally, property (ii) proved in
    Step~\textbf{2} implies $s^{\varepsilon}(t,x)\to\tilde s(t,x)$
  a.e.~in $\Omega$, completing the proof.\end{proof}

\begin{proof}[Proof of Theorem~\ref{thm:main}]
  By Theorem~\ref{thm:strongconvergence} we know that there exists a
  subsequence of wave front tracking approximate solutions constructed
  in Section~\ref{sec:FT}
  $\left(s^{\varepsilon},c^{\varepsilon},k^{\varepsilon}\right)$
  which converges strongly in $L^{1}_{loc}\left(\Omega\right)$
  to a limit $\left(\tilde s,\tilde c,\tilde k\right)$. Clearly
  $\tilde k_{t}=0$.
  Let $\phi$ be a test function with compact support in
  $\left[0,+\infty\right[\times \mathbb{R}$. By construction
  (see Section~\ref{sec:FT}) the approximate solutions satisfy
  \begin{align}
    \nonumber
    \int_{\Omega} \left[s^{\varepsilon}\phi_{t}+
    F^{\varepsilon}(t,x,s^{\varepsilon})\phi_{x}\right](t,x)\;
    dtdx
    +\int_{\mathbb{R}}\bar s^{\varepsilon} (x)\phi\left(0,x\right)\; dx =0,\\\nonumber
    \int_{\Omega} \left[c^{\varepsilon}s^{\varepsilon}
    \phi_{t}+c^{\varepsilon}F^{\varepsilon}(t,x,s^{\varepsilon})\phi_{x}\right](t,x)\;
    dtdx
    +\int_{\mathbb{R}}\bar c^{\varepsilon}\left(x\right)\bar s^{\varepsilon}
    (x)\phi\left(0,x\right)\; dx=0,\\
    \nonumber
    k^{\varepsilon}\left(t,x\right)=\bar k^{\varepsilon}(x), \quad \forall (t,x)\in \Omega.
  \end{align}
  The uniform estimate~\eqref{eq:unifEstimates} and the strong
  convergence of approximate solutions allows us to pass to the limit
  and to conclude that the limit $\left(\tilde s, \tilde c, \tilde
    k\right)$ satisfies Definition~\ref{def:main}.
\end{proof}

\bigskip

\noindent\textbf{Acknowledgment:} The present work was supported by
the PRIN~2015 project \emph{Hyperbolic Systems of Conservation Laws
  and Fluid Dynamics: Analysis and Applications} and by GNAMPA 2019
project \emph{Equazioni alle derivate parziali di tipo iperbolico o non locale ed applicazioni.}.
The authors  would like to thank the anonymous referee
for carefully reading the manuscript
and providing many useful suggestions.

  \bibliography{PFCC-CMS-Revised}

    \end{document}